\documentclass{amsart}

\usepackage{microtype}
\usepackage{booktabs}
\usepackage{graphicx}
\usepackage{subfigure}
\usepackage{hyperref}

\newtheorem{thm}{Theorem}[section]

\newtheorem{lem}[thm]{Lemma}

\theoremstyle{definition}
\newtheorem{defn}[thm]{Definition}
\newtheorem{conv}[thm]{Convention}
\newtheorem*{example}{Example}

\theoremstyle{remark}
\newtheorem{rmk}[thm]{Remark}

\begin{document}

\title[Boundary-twisted normal form]{Boundary-twisted normal form and the number of elementary moves to unknot}
%\title{A short proof of the Hass--Lagarias theorem}
\author{Chan-Ho Suh}
\address{
Bard High School Early College \\
30-20 Thomson Ave\\
Long Island City, NY 11101
}
\email{chanhosuh@gmail.com}
\thanks{Research was partially funded by the National Science Foundation (VIGRE DMS-0135345 and DMS-0636297).}


\date{\today}

\begin{abstract}
Suppose $K$ is an unknot lying in the 1-skeleton of a triangulated $3$-manifold with $t$ tetrahedra.  Hass and Lagarias showed there is an upper bound, depending only on $t$, for the minimal number of elementary moves to untangle $K$.  We give a simpler proof, utilizing a normal form for surfaces whose boundary is contained in the 1-skeleton of a triangulated 3-manifold.  We also obtain a  significantly better upper bound of $2^{120t+14}$ and improve the Hass--Lagarias upper bound on the number of Reidemeister moves needed to unknot to $2^{10^5 n}$, where $n$ is the crossing number.
\end{abstract}

\maketitle

\section{Introduction}
Suppose $M$ is a triangulated, compact 3-manifold with $t$ tetrahedra and $K$ is an unknot in the 1-skeleton.  Recall that $K$ can be isotoped in $M$ using polygonal moves across triangles called \emph{elementary moves}.  J. Hass and J. Lagarias  obtained an upper bound of $2^{10^7t}$ on the minimum number of elementary moves to take $K$ to a triangle in one tetrahedron \cite{hass-lagarias2001}.

The central idea of their proof is to use normal surface theory.  Take a double barycentric subdivision of $M$, and let $N(K)$ denote the simplicial neighborhood of $K$.  Since $K$ is the unknot, a homotopically nontrivial curve $l$ on $\partial N(K)$ bounds a normal disc $D$ in $M - \operatorname{int}(N(K))$.  Normal surface theory gives an exponential upper bound in $t$ on the number of triangles in a minimal such $D$ \cite{hass-lagarias-pippenger1999}.  Naively we might think to move $K$ across $D$ and obtain the bound that way, but this overlooks that $K$ must first be moved to $l = \partial D$.  Thus Hass and Lagarias break up the complete isotopy of $K$ to a triangle into three parts.  First isotope $K$ to $\partial N(K)$, isotope across $\partial N(K)$ to $l$, and finally isotope across $D$.  The bound on $D$ gives a similar bound for the number of elementary moves to isotope across $D$.  The bounds for the number of elementary moves realizing the other isotopies are obtained from an involved analysis, which takes up a good part of \cite{hass-lagarias2001}.  Independently, but also using normal surface theory, S. Galatolo worked out a bigger bound, but in his announcement the problem of how to maneuver $K$ to $l$ is not adequately addressed \cite{galatolo1998}.  

By considering a normal form for surfaces whose boundary is contained in the 1-skeleton of a triangulated 3-manifold, we avoid the retriangulation and subsequent isotopies related to $N(K)$.  This not only substantially simplifies the proof but improves the upper bound on the number of moves to $2^{120t+14}$.

A corollary of bounding the number of elementary moves is a bound on the number of Reidemeister moves to unknot an unknot diagram with crossing number $n$.  Our result lets us improve the bound given by Hass and Lagarias from $2^{10^{11}n}$ to $2^{10^5n}$.

In section~\ref{sec:btnf} we define the normal form and relate it to a restricted version of normal surface theory in truncated tetrahedra.  Section~\ref{sec:enumeration} explains how to enumerate the normal discs, which is a key technical detail for this paper.  In section~\ref{sec:normalization} we explain how an essential surface spanning a link in the 1-skeleton can be isotoped into boundary-twisted normal form.  Section~\ref{sec:fund} proves the existence of a fundamental spanning disc for the unknot.  In Section~\ref{sec:disc_bound}, using the previous results, we obtain an improved upper bound on the number of triangles needed in a spanning disc for an unknot.  From this bound, we obtain bounds which improve the main results of \cite{hass-lagarias2001}.  
%Finally, in Section~\ref{sec:comparison}, we explain in greater detail the difference between our method and the Hass--Lagarias method and give another approach to reduce the Hass--Lagarias bound on Reidemeister moves.

\section{Boundary-twisted and boundary-restricted normal forms}\label{sec:btnf}

\subsection{Marked triangulations}
Given a pair $(M, L)$ consisting of a compact 3--manifold $M$ and link $L$ in $M$, we say $\mathcal T$ is a \emph{marked triangulation} for $(M, L)$ if $\mathcal T$ is a triangulation of $M$ with $L$ contained in the 1--skeleton.  The vertices and edges of $\mathcal T$ which belong to $L$ are called \emph{marked vertices and edges}.

If two edges of the same face of a tetrahedron are marked and not part of a triangle component of $L$, there is an isotopy of these edges to the third edge of the face reducing the number of marked edges.  It will suffice to consider that $L$ has no triangle component and the number of edges is minimal.

\begin{conv}
No component of $L$ is a triangle, and in any tetrahedron of a marked triangulation, every face has at most one marked edge.
\end{conv}

\subsection{Boundary-twisted normal form}
A tetrahedron, $T$, of a marked triangulation is of 9 possible types and is called a \emph{marked tetrahedron}.  Normal disc types in a marked tetrahedron are more complicated than the usual triangles and quads.

%%%%%%%%% Normal arcs in a face %%%%%%
\begin{figure}[htbp]
\begin{center}
\subfigure[edge-edge arc \label{ee_arc}]{
\includegraphics[width=2.6cm]{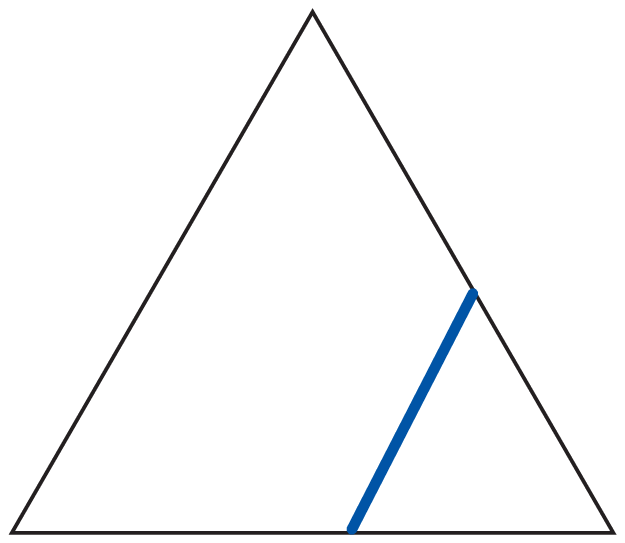}
}
\hspace{.3cm}
\subfigure[vertex-edge arc \label{ve_arc}]{
\includegraphics[width=2.6cm]{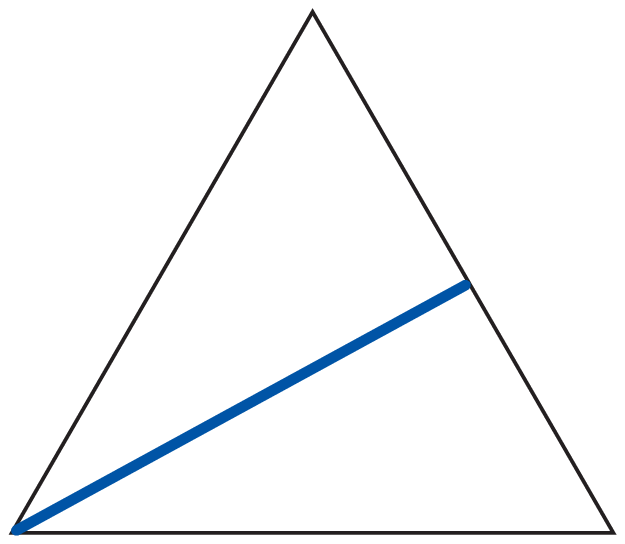}
}
\hspace{.3cm}
\subfigure[vertex-vertex arc \label{vv_arc}]{
\includegraphics[width=2.6cm]{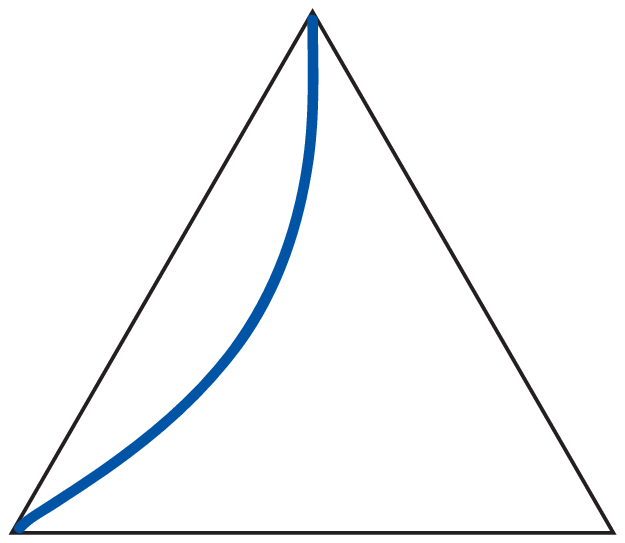}
}
\caption{Normal arcs in a face}
\label{fig:normal_arc-bdy}
\end{center}
\end{figure}
%%%%%%%%%%%%%%%%%%%%%%

\begin{defn}
We define a \emph{normal arc} in a face of $T$ to be a properly embedded arc such that one of the following holds:
\begin{itemize}
\item starts in the interior of an edge and ends in the interior of another edge (Figure~\ref{ee_arc})
\item starts at a marked vertex and ends in the interior of the opposite edge (Figure~\ref{ve_arc})
\item starts at a marked vertex and ends at a different marked vertex (Figure~\ref{vv_arc})
\end{itemize}
\end{defn}

\begin{defn}\label{defn:twisted}
We define a \emph{twisted normal disc} $D$ in a marked tetrahedron $T$ to be a properly embedded disc such that its boundary $\partial D$ satisfies the following conditions ($e$ is an edge and $\Delta$ is a face):
\begin{enumerate}
\item $\partial D \cap \operatorname{int}(\Delta)$ consists of normal arcs for every face $\Delta$
\item $\partial D \cap e$ is one of the following: a) empty, b) one endpoint of $e$, c) an arc and $e$ is marked, d) an interior point of $e$ and $e$ is unmarked, e) both endpoints of $e$ and $e$ is unmarked 
\item a pair of normal arcs of $\partial D$ abutting the same point or arc of $\partial D \cap e$ must be in different faces
\item $\partial D \cap \Delta$ does not have two normal arcs with endpoints at the same marked vertex 
\end{enumerate}
\end{defn}

\begin{defn}[Boundary-twisted normal form]
Let $(M, L)$ be a 3-manifold with link $L$ and marked triangulation $\mathcal T$.  Suppose also $S$ is a surface with boundary contained in $L$.  We say $S$ is in \emph{boundary-twisted normal form} if its intersection with every tetrahedron consists of twisted normal discs.
\end{defn}

In attempting normal surface theory with boundary-twisted normal surfaces, a technical obstacle arises: it is necessary to consider an additional kind of surface punctured by marked edges.  Rather than work with that setup, we have chosen to work in a more familiar setting.  We will truncate each marked tetrahedron at its marked vertices and edges.  Then we do normal surface theory with respect to these truncated tetrahedra while imposing some additional conditions on the surface's boundary behavior in the truncated regions.

\begin{defn}
Let $T$ be a marked tetrahedron.  Then the truncation of $T$, denoted $T_{tr}$, is a polyhedron obtained by the process indicated in Figure~\ref{fig:truncate}.  Given a marked triangulation $\mathcal T$, the \emph{truncated triangulation} $\mathcal T_{tr}$ is the polyhedral decomposition given by truncating every  marked tetrahedron of $\mathcal T$.
\end{defn}

\begin{figure}
\begin{center}
\subfigure[Start by truncating (cut off) marked edges]
{\includegraphics[height=4.0cm]{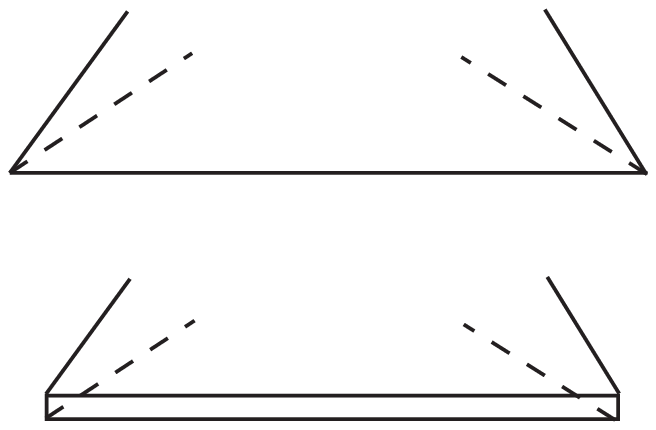}}
\hspace{.6cm}
\subfigure[Then truncate marked vertices]
{\includegraphics[width=2.2cm]{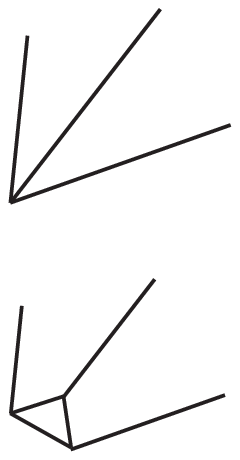}}
\end{center}
\caption{Obtaining truncated tetrahedra from marked tetrahedra}
\label{fig:truncate}
\end{figure}

\begin{defn}
A properly-embedded disc $D$ in a truncated tetrahedron $T_{tr}$ is \emph{normal} if it is the restriction to $T_{tr}$ of a twisted normal disc in $T$.  A surface in a truncated triangulation is normal if it intersects each triangulation in normal discs.
\end{defn}

\begin{rmk}
This is not the same as the usual generalization of the concept of normal disc from a tetrahedron to a polyhedron, e.g. \cite{brittenham1991} or \cite{lackenby2001b}.  Our definition follows naturally from the view of boundary-twisted normal surfaces and has the advantage of reducing the number of disc types.
\end{rmk}

Since the truncation $T_{tr}$ can be considered a subset of $T$, a normal disc in $T_{tr}$ naturally sits inside $T$.  It can be extended to a twisted normal disc of $T$ in a natural way.  In general, this extension is far from unique, since one can choose different directions to twist and in some cases twist either in the interior of an edge or at an end (see Figure~\ref{fig:twisting}).

\begin{figure}
\begin{center}
\subfigure[In the interior of a marked edge]
{\includegraphics[width=4.8cm]{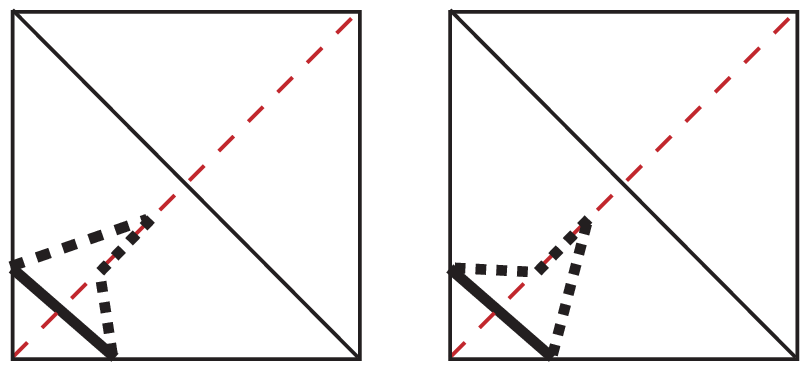}}
\hspace{.5cm}
\subfigure[At an end of a marked edge]
{\includegraphics[width=4.8cm]{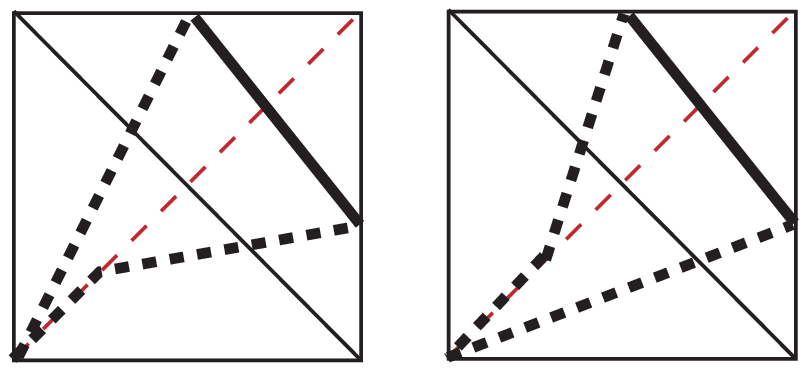}}
\caption{Discs differing by an opposite twist along a marked edge}
\label{fig:twisting}
\end{center}
\end{figure}

\subsection{Boundary-restricted normal form}
A boundary-twisted normal surface is a normal surface in the truncated triangulation $\mathcal T_{tr}$, whose boundary has been extended through the truncated region to $L$.  In order for this extension to happen in a simple way, the normal surface must satisfy additional conditions on its boundary.  We call this type of normal surface in $\mathcal T_{tr}$ a \emph{boundary-restricted normal surface}.  

There are two kinds of boundary regions given by the truncation, triangles and rectangles.  Rectangles should be visualized as ``long'' with the long sides corresponding to the longitudal direction and the short sides corresponding to the meridional direction.

We define a boundary-restricted normal surface to be a normal surface $S$ in $\mathcal T_{tr}$ such that for any rectangle $R$  we require $\partial S \cap R$ look as in Figure~\ref{fig:rectangle}.  We either have one longitudal normal arc, $n$ meridional arcs, or $n$ meridional arcs with one or two corner-cutting arcs ($n \geq 0$).  If there are two corner-cutting arcs, they must be at opposite corners.

\begin{figure}
\begin{center}
\subfigure[One longitudal arc]
{\includegraphics[width=3cm]{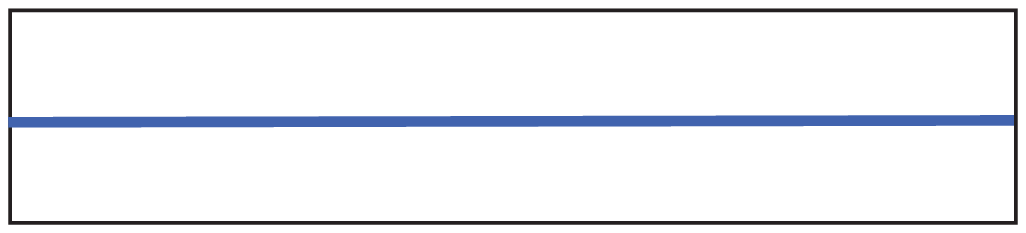}}
\hspace{.5cm}
\subfigure[$n$ meridional arc(s)]
{\includegraphics[width=3cm]{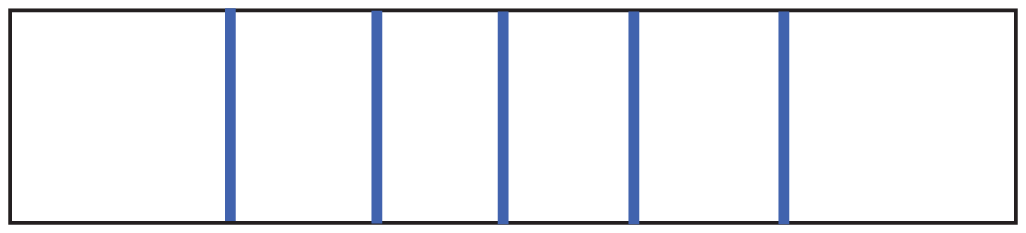}}\\
\subfigure[$n$ meridional arc(s) and one corner-cutting arc]
{\includegraphics[width=3cm]{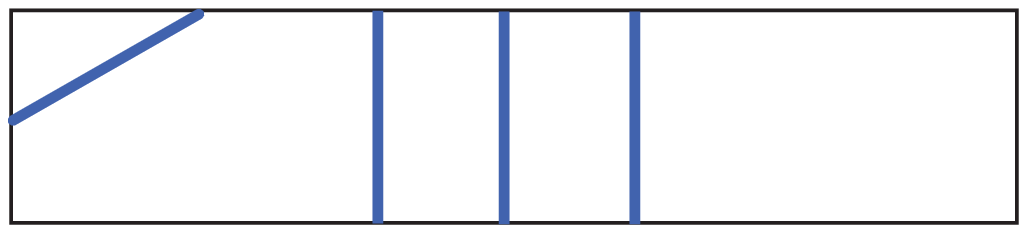}}
\hspace{.5cm}
\subfigure[$n$ meridional arc(s) and two corner-cutting arcs]
{\includegraphics[width=3cm]{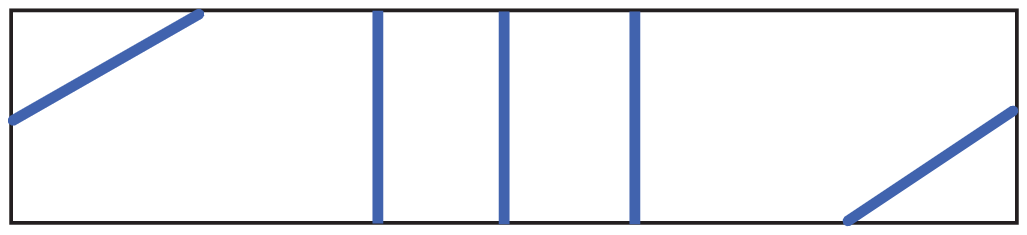}}
\caption{Pictures of the restrictions on the boundary of a boundary-restricted normal surface}
\label{fig:rectangle}
\end{center}
\end{figure}

\section{Enumeration of twisted normal discs}\label{sec:enumeration}

\begin{figure}
\begin{center}
\includegraphics[width=6cm]{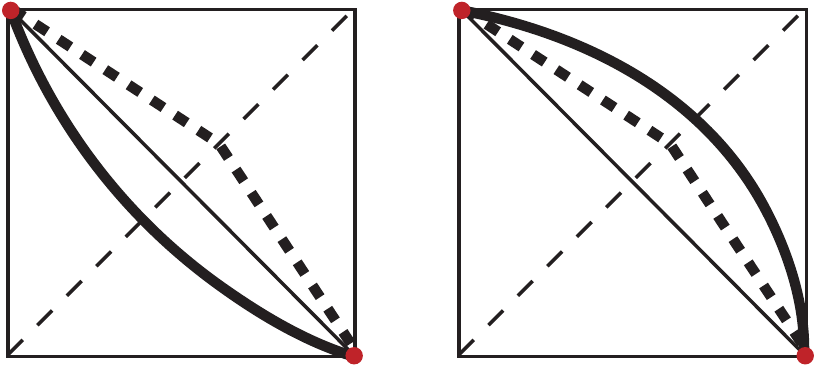}
\caption{Discs that differ from flipping a vertex-vertex arc to an adjacent face}
\label{fig:flipping}
\end{center}
\end{figure}

Figures~\ref{fig:tet_nv0e}-\ref{fig:tet_0v2e} show pictures of some twisted normal discs in the different marked tetrahedra.  To obtain further disc types from a particular figure, apply a symmetry of the tetrahedron (preserving markings), change the direction of twisting along a marked edge (Figure~\ref{fig:twisting}), and/or flip a vertex-vertex arc to an adjacent face (Figure~\ref{fig:flipping}).  Each figure's caption includes a number denoting the total number of disc types obtainable from a figure in this manner.  Also note that some of the discs in one kind of marked tetrahedron also show up in other kinds, but the figures illustrate only the non-redundant ones.  For example, all of the disc types for a tetrahedron with one marked vertex and one marked edge show up in the tetrahedron with two marked vertices and one marked edge.

Except for these operations and the non-illustration of redundant disc types, these figures are complete.  The organization of the figures is meant to suggest the method of enumeration.  We now explain the enumeration in further detail. 

\subsection{Marked tetrahedron with only marked vertices}
For an unmarked tetrahedron, the only twisted normal discs are the standard normal discs: triangles and quads (Figure~\ref{fig:usual}).  This gives the usual 7 disc types.

For a marked tetrahedron with one or more marked vertices but no marked edges, in addition to the previous triangles and quads, we have vertex touching triangles (Figure~\ref{fig:vertex_touching}) and possibilities utilizing vertex-vertex arcs (Figures~\ref{fig:bigon}, \ref{fig:triangle_vv_arc}, \ref{fig:triangle_vv_arc2}, and \ref{fig:quad_vv_arc}).  

For a tetrahedron with one marked vertex, we have 7 discs from the unmarked case and the three vertex touching triangles, making 10 disc types.  

For a tetrahedron with two marked vertices, we have 4 quads, 3 triangles, 6 vertex touching triangles, 2 triangles with a vertex-vertex arc, and 1 bigon.  This gives 16 total.  

For a tetrahedron with three marked vertices, 4 quads, 3 triangles, 9 vertex touching triangles, 6 triangles with one vertex-vertex arc, 4 triangles with all vertex-vertex arcs, and 3 bigons.  This gives 29 total. 

For a tetrahedron with four marked vertices, we have 4 triangles, 3 quads, 6 quads with all vertex-vertex arcs, 12 vertex-touching triangles, 12 triangles with one vertex-vertex edge, 16 triangles with all vertex-vertex arcs, 6 bigons, giving 59 total.

\begin{figure}[htbp]
\begin{center}
\subfigure[The usual suspects: triangle (4) and quad (3) \label{fig:usual}]
{\includegraphics[width=6.5cm]{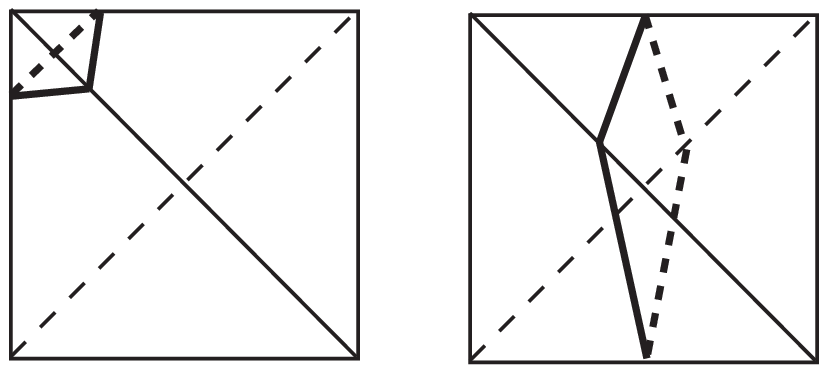}}\\
\subfigure[A vertex-touching triangle (3) \label{fig:vertex_touching} ]
{\includegraphics[width=3cm]{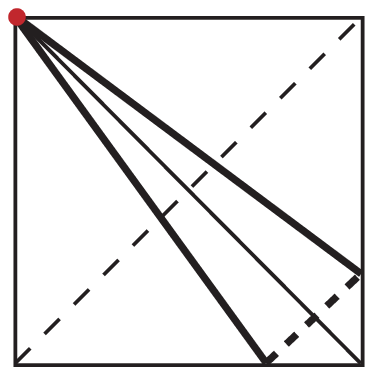}}
\hspace{.3cm}
\subfigure[A bigon (1)\label{fig:bigon} ]
{\includegraphics[width=3cm]{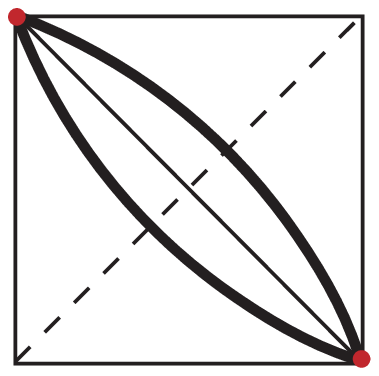}}
\hspace{.3cm}
\subfigure[Triangle with one vertex-vertex arc (2)\label{fig:triangle_vv_arc} ]
{\includegraphics[width=3cm]{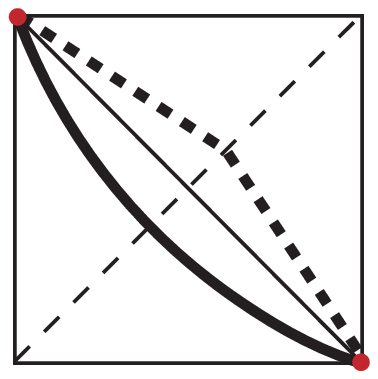}}\\
\subfigure[Triangle with all vertex-vertex arcs (4) \label{fig:triangle_vv_arc2} ]
{\includegraphics[width=3cm]{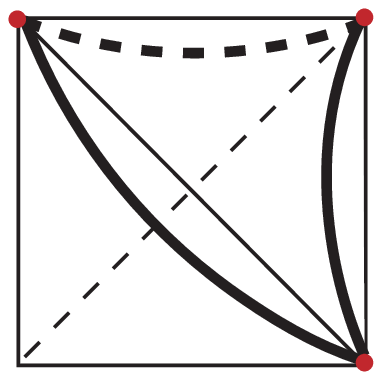}}
\hspace{.3cm}
\subfigure[Quad with all vertex-vertex arcs (6) \label{fig:quad_vv_arc} ]
{\includegraphics[width=3cm]{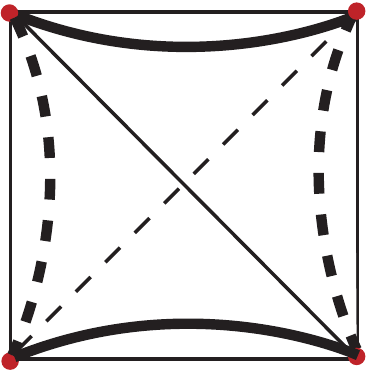}}
\hspace{.3cm}
\caption{Twisted normal discs in a tetrahedron with $n$ marked vertices ($n \geq 0$)}
\label{fig:tet_nv0e}
\end{center}
\end{figure}

\subsection{Tetrahedron with one marked edge}
For the one marked edge case, we have 30 total disc types.  

There are two triangles and one quad that don't use the marked edge at all.  Moving on to discs that utilize an endpoint of the marked edge, there are 6 vertex-touching triangles, as in a previous case.   

The remaining discs utilize the entire marked edge or a subarc of it (see Figure~\ref{fig:tet_0v1e}).  The first row illustrates those using the entire marked edge, while the last two rows show disc types using an interior or exterior subarc, resp.

\begin{figure}
\begin{center}
\subfigure[1\label{fig:triangle_full}]
{\includegraphics[width=3cm]{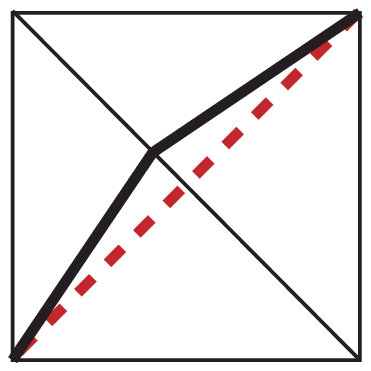}}\\
%\hspace{.3cm}
%\subfigure[2\label{fig:pentagon_full}]
%{\includegraphics[width=3cm]{pentagon_full_marked_edge}}\\
\subfigure[4 \label{fig:0v1e-int_quad}]
{\includegraphics[width=3cm]{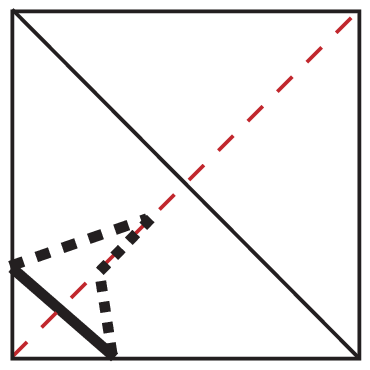}}
\hspace{.3cm}
\subfigure[4\label{fig:0v1e-int_pent}]
{\includegraphics[width=3cm]{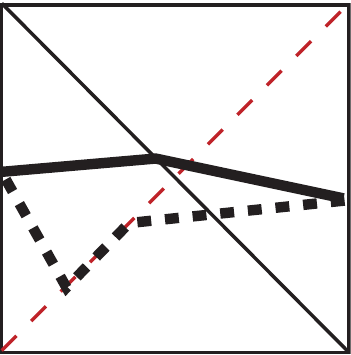}}\\
\subfigure[4\label{fig:0v1e-ext_quad}]
{\includegraphics[width=3cm]{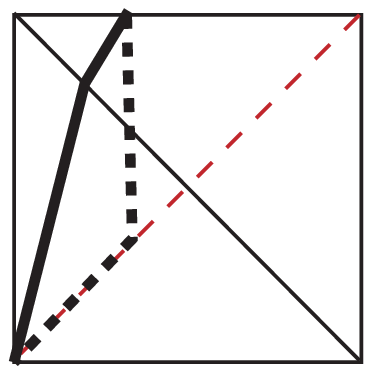}}
\hspace{.3cm}
\subfigure[4\label{fig:0v1e-ext_quad2}]
{\includegraphics[width=3cm]{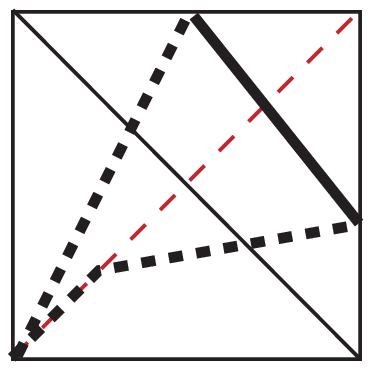}}
\hspace{.3cm}
\subfigure[4\label{fig:0v1e-ext_pent}]
{\includegraphics[width=3cm]{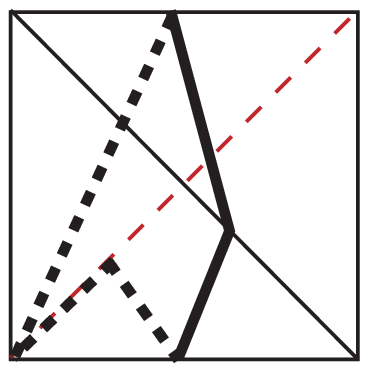}}
\caption{Twisted normal discs in a tetrahedron with one marked edge}
\label{fig:tet_0v1e}
\end{center}
\end{figure}

\subsection{Tetrahedron with one marked edge and one marked vertex}
There are 47 total disc types, with 30 types from the 1 marked edge case, and 17 new types, which we will enumerate and illustrate.  The 17 new disc types must utilize the marked vertex (see Figure~\ref{fig:tet_1v1e}).  The first row of the figure are 2 discs not utilizing a subarc of the marked edge: a vertex-touching triangle and a bigon.  Each subsequent row of the figure, as before, illustrates the disc types utilizing a particular kind of subarc of the marked edge.  

\begin{figure}
\begin{center}
\subfigure[1\label{fig:1v1e-vertex-touching}]
{\includegraphics[width=3cm]{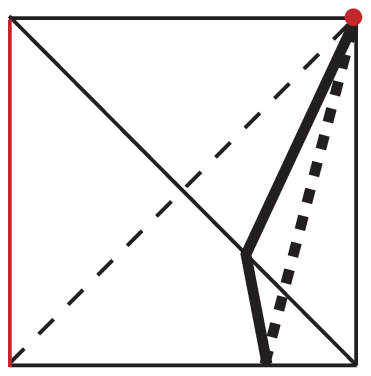}}
\subfigure[1\label{fig:1v1e-bigon}]
{\includegraphics[width=3cm]{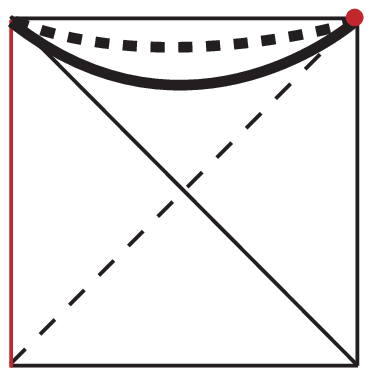}}\\
\subfigure[3\label{fig:1v1e-triangle_full}]
{\includegraphics[width=3cm]{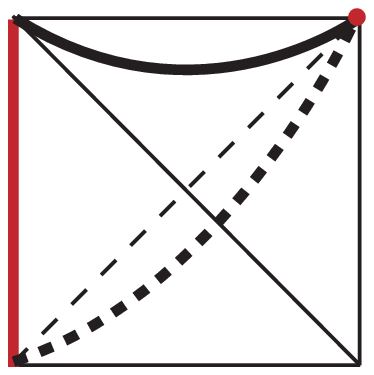}}
%\subfigure[2\label{fig:1v1e-quad_full}]
%{\includegraphics[width=3cm]{1v1e-quad_full_edge}}\\
\subfigure[2\label{fig:1v1e-ext_triangle}]
{\includegraphics[width=3cm]{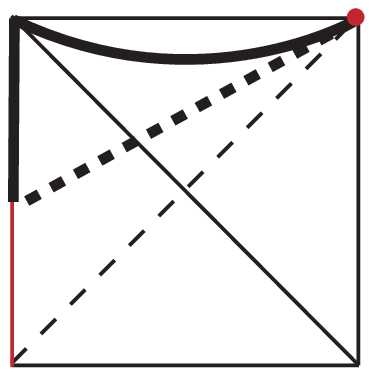}}
\subfigure[4\label{fig:1v1e-ext_quad}]
{\includegraphics[width=3cm]{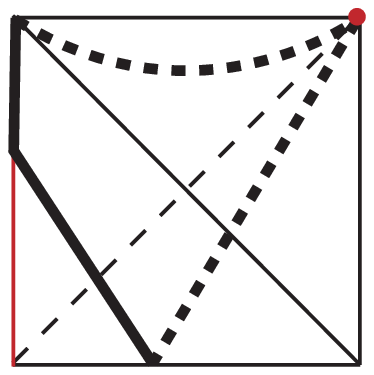}}
%\subfigure[2\label{fig:1v1e-ext_quad2}]
%{\includegraphics[width=3cm]{1v1e-exterior_quad2}}
\subfigure[2\label{fig:1v1e-ext_quad3}]
{\includegraphics[width=3cm]{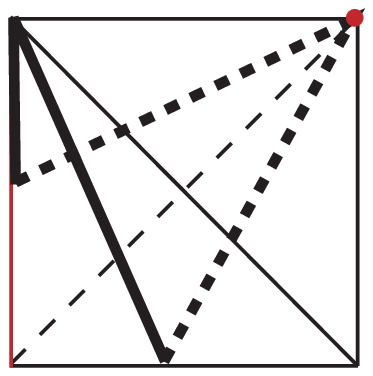}}\\
\subfigure[4\label{fig:1v1e-int_quad}]
{\includegraphics[width=3cm]{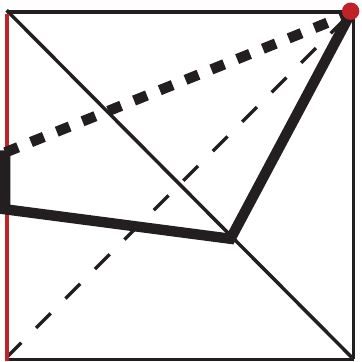}}
\caption{Twisted normal discs in a tetrahedron with one marked edge and one marked vertex}
\label{fig:tet_1v1e}
\end{center}
\end{figure}

\subsection{Tetrahedron with one marked edge and two marked vertices}
There are a total of 93 disc types.  From the 1 marked edge case, we have 30.  From the 1 marked edge and 1 marked vertex case, we get double contributions (since we now have two marked vertices), giving 2*17 = 34.  We enumerate the 29 new disc types below.  Note that they must utilize both of the marked vertices.

There are 9 disc types not using any arc of the marked edge: 1 bigon (Figure~\ref{fig:2v1e-bigon}) and 8 triangles with all vertex-vertex arcs (Figure~\ref{fig:2v1e-triangle_vv_arcs}).    

The remaining disc types are all quads.  There are 4 quads with three vertex-vertex arcs which utilize the full marked edge (Figure~\ref{fig:2v1e-quad_full_edge}).  There are 12 quads utilizing an exterior subarc of the marked edge (Figure~\ref{fig:2v1e-ext_quad}).  There are 4 quads utilizing an interior subarc of the marked edge (Figure~\ref{fig:2v1e-int_quad}).

\begin{figure}[htbp]
\begin{center}
\subfigure[1\label{fig:2v1e-bigon}]
{\includegraphics[width=3cm]{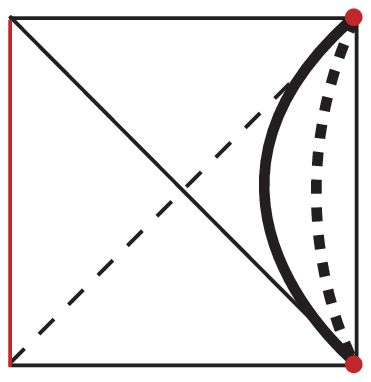}}
\hspace{.3cm}
\subfigure[8\label{fig:2v1e-triangle_vv_arcs}]
{\includegraphics[width=3cm]{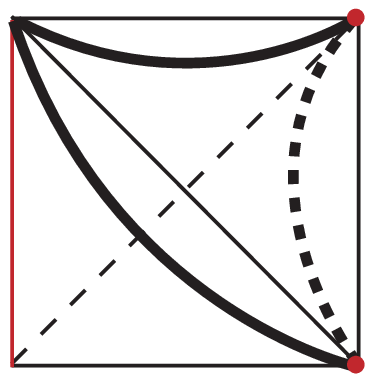}}
\hspace{.3cm}
\subfigure[4\label{fig:2v1e-quad_full_edge}]
{\includegraphics[width=3cm]{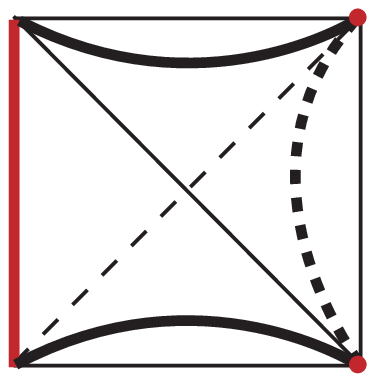}}\\
\subfigure[12\label{fig:2v1e-ext_quad}]
{\includegraphics[width=3cm]{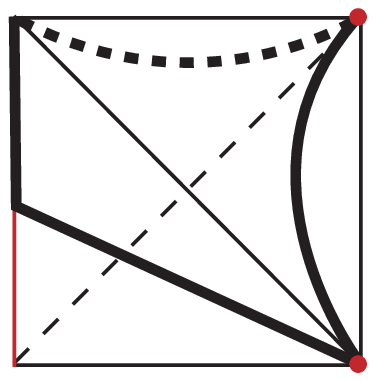}}
\hspace{.3cm}
\subfigure[4\label{fig:2v1e-int_quad}]
{\includegraphics[width=3cm]{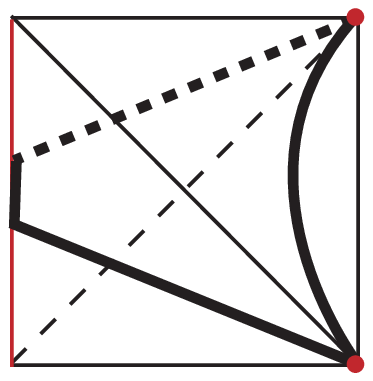}}
\caption{Twisted normal discs in a tetrahedron with marked edge and two marked vertices}
\label{fig:tet_2v1e}
\end{center}
\end{figure}

\begin{figure}[htbp]
\subfigure[8\label{fig:truncate_disc}]
{\includegraphics[width=3cm]{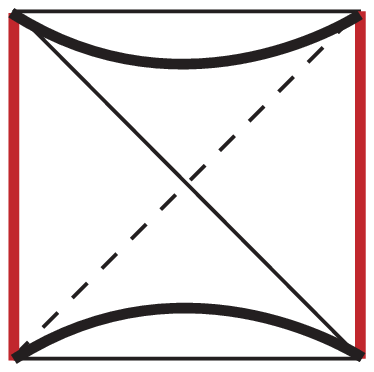}}\hspace{1cm}
\subfigure[16]
{\includegraphics[width=3cm]{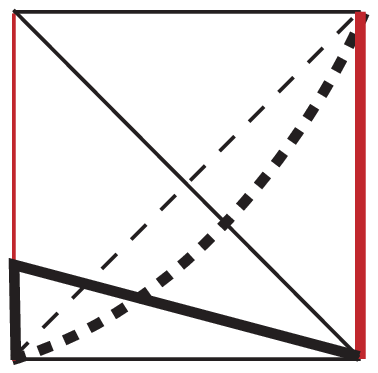}}\hspace{1cm}
\subfigure[4]
{\includegraphics[width=3cm]{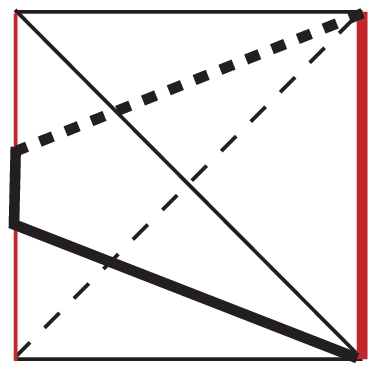}}\\
\subfigure[2\label{fig:truncate_disc2}]
{\includegraphics[width=3cm]{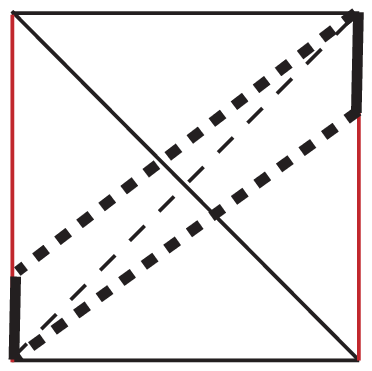}}\hspace{1cm}
\subfigure[4]
{\includegraphics[width=3cm]{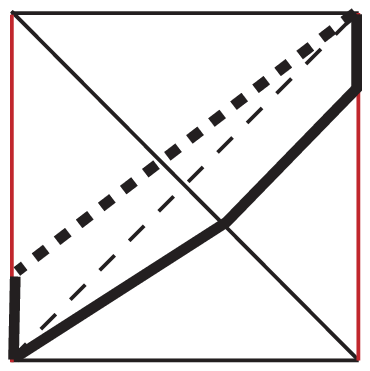}}\hspace{1cm}
\subfigure[8]
{\includegraphics[width=3cm]{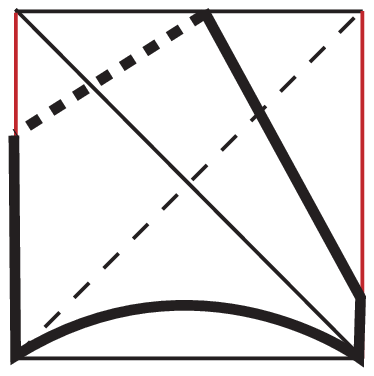}}\\
\subfigure[16]
{\includegraphics[width=3cm]{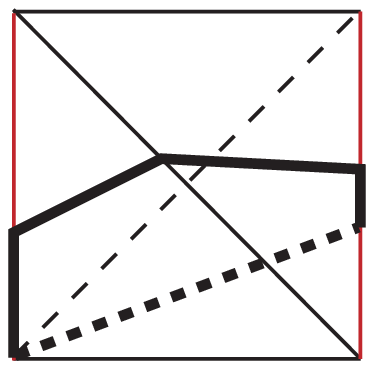}}\hspace{1cm}
\subfigure[4]
{\includegraphics[width=3cm]{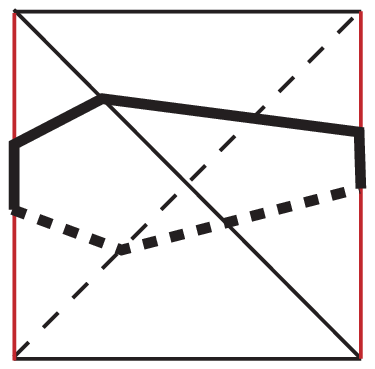}}\hspace{1cm}
\subfigure[4\label{fig:truncate_disc3}]
{\includegraphics[width=3cm]{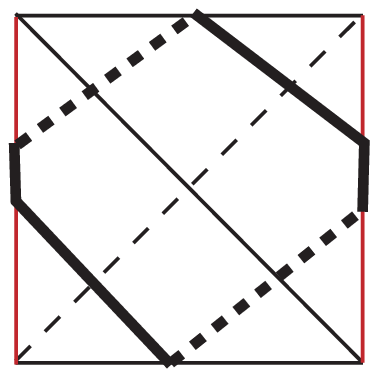}}
\caption{Twisted normal discs in a tetrahedron with two marked edges}
\label{fig:tet_0v2e}
\end{figure}

\subsection{Tetrahedron with two marked edges}
There are a total of 148 disc types.  There are 66 new disc types, which are illustrated in Figure~\ref{fig:tet_0v2e}.  

As before, the discs from the previous cases should be counted, but not all of them are compatible with this type of marked tetrahedron.  From the 1 marked edge case, we have 1 normal quad, 4 vertex-touching triangles, 8 quads with an interior subarc, and 8 quads with an exterior subarc.  The total is 21. 

From the 1 marked edge and 1 marked vertex case, we have 2 bigons, and a quadruple contribution of the rest of the disc types, except for the vertex-touching triangles which are considered in the previous paragraph.  This gives 2 + 4*15 = 62 total. 

The non-redundant discs from the 1 marked edge and 2 marked vertex case are not allowable here.

\subsection{Obtaining normal discs from twisted normal discs by truncating}
For normal surface theory, we need to know the maximal number of normal discs in any truncated tetrahedron.  
Table~\ref{big_table} shows the total number of normal discs in a truncated tetrahedron.  These are obtained by truncating the tetrahedron and observing carefully how different twisted normal disc types become the same normal disc type.  For example, after truncation of the edges, the two discs illustrated in Figure~\ref{fig:truncate_disc2} are amongst the discs obtained from Figure~\ref{fig:truncate_disc}.  The last column of the table will be used later in Section~\ref{sec:disc_bound}.  It is the maximum number of disc types whose boundary has a common normal arc in a face.  

%%%%%%%%%%%%%%%%%%%%%%%%%%%%%%%%%%
%% Table of boundary-twisted normal discs

\begin{table}[htbp]
\begin{center}
\begin{tabular}{@{}lr@{}}
\toprule
\textbf{Tetrahedron type} &  \textbf{Total}\\
\midrule
No marked vertices or edges & 7\\
one marked vertex & 10\\
two marked vertices &  16\\
three marked vertices &  29\\
four marked vertices &  59\\ 
1 marked edge & 30 \\ 
1 marked edge, 1 marked vertex & 47 \\

1 marked edge, 2 marked vertices & 93 \\ 

2 marked edges &  148\\
\bottomrule\\
\end{tabular}
\end{center}
\caption{Number of twisted normal disc types in each type of marked tetrahedron}
\label{table-twisted_normal}
\end{table}

%%%%%%%%%%%%%%%%%%
% The no-longer-so-big table of info

\begin{table}[htbp]
\begin{center}
\begin{tabular}{@{}lrc@{}}
\toprule
\textbf{Tetrahedron type} & \textbf{Total} & \quad \textbf{Max arc \#}\\
\midrule
No truncation & 7 & 2\\ 
One truncated vertex & 10 & 3\\ 
Two truncated vertices & 16 & 3\\ 
Three truncated vertices & 29 & 3\\ 
Four truncated vertices & 59 & 3\\ 
One truncated edge & 15 & 3\\ 
Two truncated edges & 17 & 6\\ 
One truncated edge, one truncated vertex & 22 & 3\\ 
One truncated edge, two truncated vertices & 40 & 6\\
\bottomrule\\
\end{tabular}
\end{center}
\caption{Number of normal disc types in each type of truncated tetrahedron}
\label{big_table}
\end{table}

%The following observation connects the two normal forms:

%\begin{thm}
%Let $(M, L)$ have a marked triangulation.  An incompressible surface $S$ with boundary in $L$ is in boundary-twisted normal form with respect to $\mathcal T$ if and only if $S' = S \cap  \mathcal T_{tr}$ is in boundary-restricted normal form with respect to $\mathcal T_{tr}$.
%\end{thm}

%\begin{proof}
%It is simple to see from the definitions that if $S$ is in boundary-twisted normal form, $S'$ is in boundary-restricted normal form.  For the other direction, we need to use a condition such as incompressibility.

%Let $L_0$ be a component of $L$ and let $N(L_0)$ denote the truncated neighborhood of $L_0$.  The pieces of $S$ inside $N(L_0)$ are incompressible.  By definition of boundary-restricted normal form, we can find a meridian $m$ composed of meridional arcs in each rectangle such that $\partial S'$ intersects it exactly once.

%Consider a disc $D$ in $N(L_0)$ bound by $m$.  By using an innermost circle argument, we can remove all simple closed curves of intersection in $D \cap S$.  In fact, because of the condition on $\partial S'$, we have $D \cap S$ is a single arc $\beta$ with endpoints on $L_0$ and $\partial N(L_0)$.  

%Now it is simple to see that $S \cap N(L_0)$ is an annulus.  Examining how this annulus connects $L_0$ to $\partial S'$ reveals that $S$ is in boundary-twisted normal form.  
%\end{proof}
%Since the proof uses standard arguments and the theorem is not necessary for this paper, we omit it.

\section{Isotoping to boundary-twisted normal form}\label{sec:normalization}

\begin{thm}\label{thm:btnf}
Let $M$ be a compact irreducible $3$-manifold with triangulation $\mathcal T$ and $L$ be a polygonal link contained in the 1-skeleton of $\mathcal T$.  Suppose $L$ has at most one edge in each triangle of $\mathcal T$ and bounds an incompressible surface $S$.  Then $S$ can be isotoped (rel $\partial$) into boundary-twisted normal form.
\end{thm}

\begin{proof}
Isotope $S$ to be in general position with respect to the 2-skeleton of $\mathcal T$.  Then $\operatorname{Int}(S) \cap \partial\Delta^3$ for any tetrahedron $\Delta^3$ consists of simple closed curves and/or embedded (open) arcs with endpoints on marked edges or vertices.  

Consider the portions of the marked edges of $\Delta^3$ which abut pieces of $S$ inside of $\Delta^3$.  By a small isotopy of $S$, we can assume that they are arcs or endpoints.  

We call a simple closed curve in $S \cap \partial \Delta^3$ which is a boundary component of an annulus contained in $ S \cap \Delta^3$ a \emph{circle of intersection} of $S$ and $\partial\Delta^3$.  Note that strictly speaking, this is an abuse of terminology as a circle of intersection should refer to a component of the intersection which is a circle.

Define the \emph{weight} of $S$ to be the sum of the number of components of $S \cap \operatorname{int}(\Delta)$ over all faces $\Delta$ of $\mathcal T$.  We will now perform a series of weight-reducing isotopies until $S$ is in boundary-twisted normal form.  Each type of isotopy will eliminate an unwanted situation and bring $S$ closer to being in boundary-twisted normal form.  After an isotopy, we may have disturbed our work in previous stages, but we can repeat the entire simplification procedure up to that point, which drives down the weight, to ensure all the previous conditions still hold.  We will assume this is done in the following descriptions.

Consider a particular $\Delta^3$.  For each circle of intersection, starting with innermost ones, we can take the disc bound by it in $\partial \Delta^3$ and push it in slightly.  By doing so we obtain a compression disc for $S$.  Since $S$ is incompressible and $M$ is irreducible, we can isotope $S$ until its intersection with $\Delta^3$ coincides with these compression discs.  Therefore we can arrange that each circle of intersection on $\partial \Delta^3$ bounds a disc inside it.  

Consider a circle of intersection which is in the interior of a face.  We can isotope the disc of $S$ bound by the circle through the face, removing one or more circles of intersection.  We repeat this to remove any further circles of intersection inside a face. 

Now consider any circle of intersection whose restriction to a face has a non-normal curve. The arc either has both its endpoints on a marked vertex, a \emph{monogon} (Figure~\ref{monogon}), or at least one endpoint is in the interior of an edge, a \emph{$D$-curve} (Figure~\ref{D-curve}).  A monogon can be eliminated by pushing the disc it bounds into the next tetrahedron.

\begin{figure}
\begin{center}
\subfigure[monogon \label{monogon}]{
\includegraphics[width=3cm]{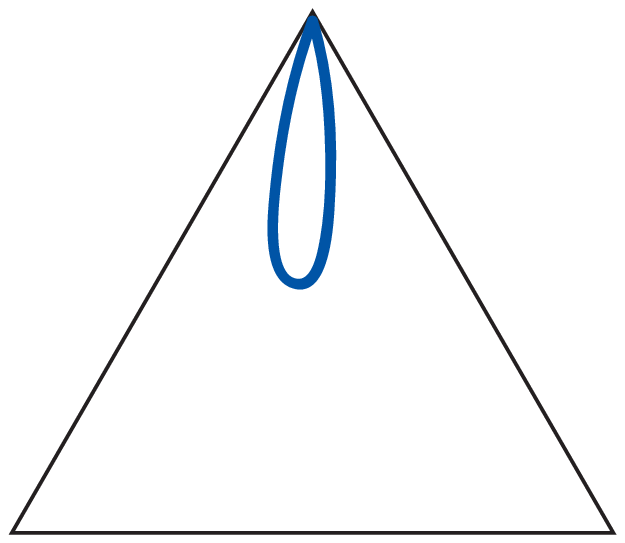}
}
\hspace{.3cm}
\subfigure[D-curve \label{D-curve}]{
\includegraphics[width=3cm]{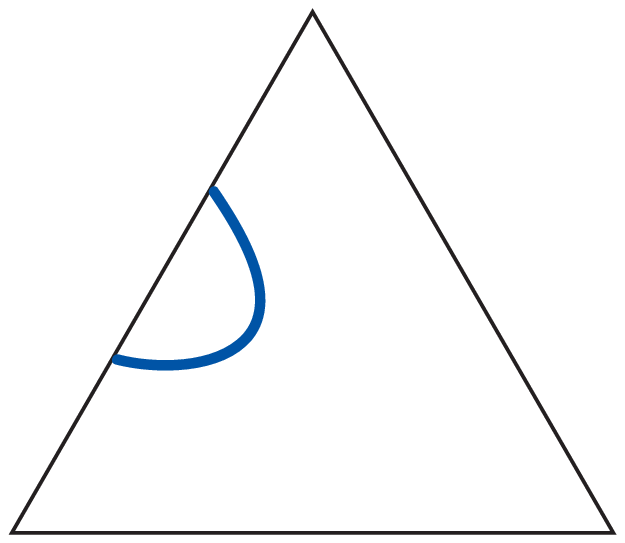}
}
\caption{Some non-normal curves}
\end{center}
\end{figure}

In the $D$-curve case, we can suppose it is innermost.  Such an innermost curve and a segment of the marked edge bounds a disc in the face.  This is a compression disc, so we can isotope $S$ to the disc and then push through the face.  This isotopy must reduce the weight.    

Because we are decreasing the weight, by repeating this procedure, we can ensure that the surface $S$ in $\Delta^3$ satisfies:
\begin{itemize}
\item Its intersection with each tetrahedron consists of discs
\item The boundary of each disc consists of normal arcs and arcs of marked edges (Definition~\ref{defn:twisted} (1) and (2c))
\end{itemize}

We can eliminate a disc with two normal arcs in a face touching the same vertex by pushing the disc through the face near the vertex (Definition~\ref{defn:twisted} (4)).  This reduces the weight by combining the two arcs into one.  

It may be the case that the boundary of a disc $D$ intersects an edge $e$ more than once, including at least once in its interior.  We can suppose $D$ is innermost.  We now have two cases: 1) $e$ is not marked. 2) $e$ is marked.  

Case 1): There is a disc $E$, whose interior is disjoint from $S$, such that $\partial E = \alpha \cup \beta$, where $\alpha$ is an arc on $S$ connecting two points of $\partial D \cap e$ and $\beta$ is a subarc of $e$.  We can push $S$ across $E$ so that $\alpha$ is taken to $\beta$.  A further small isotopy of $S$ through $e$ will split $D$ into two discs in the tetrahedron.  If the two points of $\partial D \cap e$ were in the interior of $e$, these two discs have combined weight less than that of $D$.  Otherwise, one point was at an endpoint of $e$ and it is possible the two discs have the same weight as that of $D$.  Nonetheless, the intersection of $S$ with the 1-skeleton is simplified.  Since none of our weight-reducing isotopies increase intersection with the 1-skeleton, we can eliminate this situation also. 

Case 2): If $D$ intersects $e$ only at its endpoints, we need to avoid violating Definition~\ref{defn:twisted} (2e).  In this case, there must be an arc $\alpha$ on $D$ joining the endpoints of $e$.  Then $\alpha$ and $e$ bound a compression disc which gives a weight-reducing compression.  Otherwise $D$ intersects $e$ in its interior in at least one arc.  By taking an arc $\alpha$ of the marked edge adjacent to two innermost components of $\partial D \cap e$ and an arc $\beta$ in the interior of $S$ connecting the endpoints of $\alpha$, we obtain a simple closed curve $\gamma =\alpha \cup \beta$.  $\gamma$ bounds a compression disc in $\Delta^3$ for $S$.  Thus we can isotope $S$ to the disc, reducing the weight.  

Now consider a circle of intersection which touches a marked edge but does not pass through from one face to another (this violates Definition~\ref{defn:twisted} (3)).  Suppose the curve bounds the disc $D$.  We can reduce the weight by pushing the part of $D$ near the marked edge through a face.    

Eventually we arrive at a situation where the pieces of $S$ inside $\Delta^3$ are exactly what we defined previously as twisted normal discs (see Definition~\ref{defn:twisted}).  In particular, if Definition~\ref{defn:twisted} (2) is not satisfied, clearly we can do one of the above isotopies.

Now we move onto another tetrahedron and repeat the entire process until $S$ is cleaned up to the requisite form inside the tetrahedron.  Note this may ruin our work in the previous tetrahedron.  We move onto yet another tetrahedron and do the process.  By cycling through the tetrahedra and noting that the weight is strictly decreasing, eventually the weight is at a minimum and $S$ is in boundary-twisted normal form.
\end{proof}

\section{Existence of a fundamental unknotting disc}\label{sec:fund}
Normal surface theory in truncated triangulations is analogous to that in standard triangulations.  Each normal disc type in a truncated tetrahedron is represented by a variable.  There are integer linear equations in these variables, \emph{matching equations}, given by each pair of truncated tetrahedra sharing a face.  A coordinate vector is obtained from a normal surface by simply counting the number of normal discs of each type it contains.  A normal surface's vector must be a solution to the equations, but non-negative integral solutions are not necessarily normal surfaces.  Non-negative integer solutions which satisfy a no-intersection condition are \emph{uniquely} realized as a normal surface.  The condition is that certain pairs of coordinates cannot both be nonzero; the pairs correspond to normal disc types that cannot be realized at the same time as disjoint discs.

A normal surface with vector $v$ such that $v \neq v_1 + v_2$ where each $v_i\neq 0$ is a non-negative integer solution to the matching equations is called \emph{fundamental}.  All such $v$ constitute the minimal Hilbert basis for the system of equations.  This basis is a finite, generating set for all non-negative integer solutions and can be found using methods of integer linear programming.

Let $S$ be a normal surface with solution vector $v(S) = v_1 + v_2$ and each $v_i$ is a nonnegative integer solution vector.  The no-intersection condition on $v(S)$ passes to each $v_i$ so that $v_i=v(S_i)$ for a normal surface $S$.  There is a \emph{Haken sum}, a cut-and-paste addition of normal surfaces, such that the Haken sum of $S_1$ and $S_2$ is $S$ (Figure~\ref{fig:haken_sum}).  Two important properties are that Euler characteristic is additive under Haken sum, and there is a way of cutting and pasting the ``wrong way'', the so-called \emph{irregular switch}, that creates a non-normal surface.

\begin{figure}
\begin{center}
\subfigure[Example of intersection of two summands of a normal surface]
{\includegraphics[width=3.2cm]{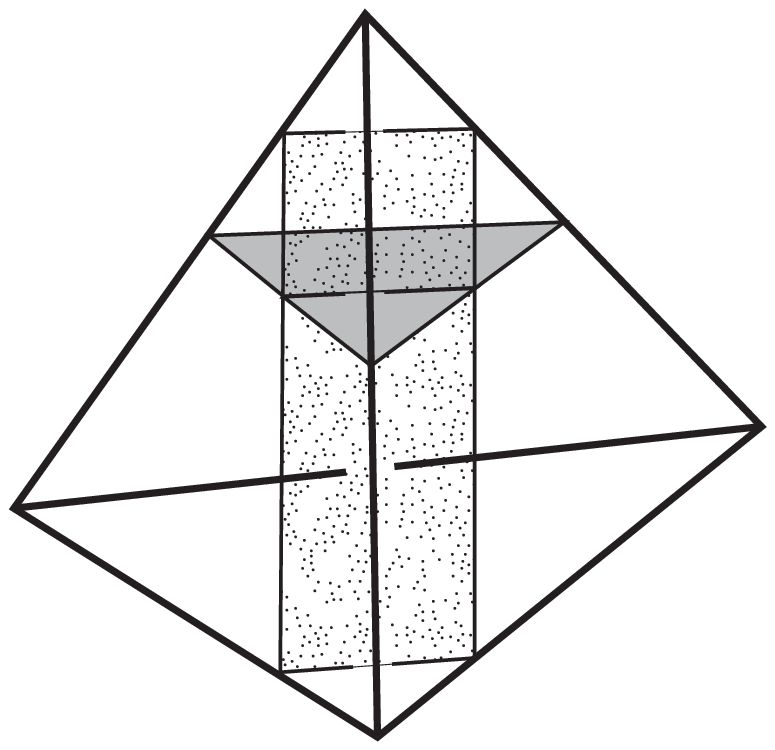}}
\hspace{.3cm}
\subfigure[A cut-and-paste giving normal discs]
{\includegraphics[width=3.2cm]{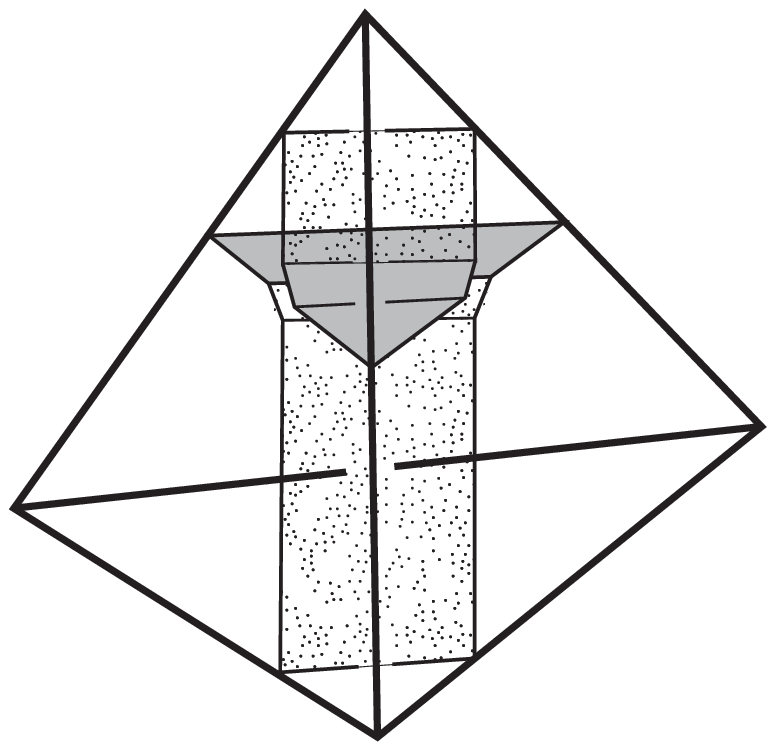}}
\hspace{.3cm}
\subfigure[A cut-and-paste resulting in non-normal discs]
{\includegraphics[width=3.2cm]{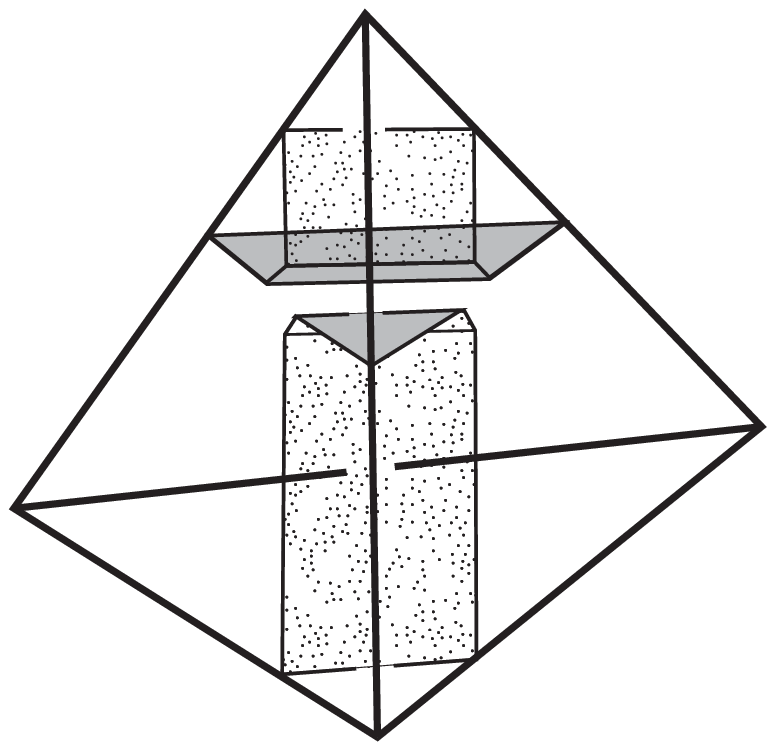}}
\caption{The Haken sum}
\label{fig:haken_sum}
\end{center}
\end{figure}

In our setup, we work with boundary-restricted normal surfaces.  In order that our arguments work, we need only ensure that passing to a summand lets us continue working with boundary-restricted normal surfaces.  That is the content of the following simple, but crucial, lemma.

\begin{lem}
Let $S$ be a boundary-restricted normal surface and suppose $v(S) = v_1 + v_2$, where each $v_i$ is a nonzero solution to the matching equations.  Then $v_i = v(S_i)$ for a unique boundary-restricted normal surface $S_i$ and $S$ is the Haken sum $S_1 + S_2$.
\end{lem}

\begin{proof}
From the remarks about normal surface theory preceding the lemma, clearly $v_i = v(S_i)$ for a unique normal surface $S_i$ and $S = S_1 + S_2$.  It remains only to check that $S_i$ is boundary-restricted.  But the conditions on the boundary (Figure~\ref{fig:rectangle}) are inherited from $v$ by each $v_i$, so this follows.
\end{proof}

Now recall that the \emph{weight} of a normal surface is the number of points of intersection with the 1-skeleton.

\begin{thm}\label{thm:fund}
Let $(M, K)$ have a marked triangulation $\mathcal T$ with $K$ a knot in $\mathcal T^{(1)}$.  Suppose $\mathcal T_{tr}$ is the truncation of $\mathcal T$ and $D$ is a a boundary-restricted normal disc in $\mathcal T_{tr}$.   If $D$ is of least weight over all such discs, then it is fundamental.
\end{thm} 

\begin{proof}
Standard techniques as in \cite{jaco-oertel1984} are applicable here.  So we are brief on some points and refer the reader to \cite{jaco-oertel1984} for more details.

Suppose $D$ is not fundamental.  Then $D = D_1 + D_2$.  By Schubert's lemma (\cite{jaco-oertel1984}, 1.9, p. 199), we can suppose the $D_i$'s are connected.  The condition on $\partial D$ means that one summand, say $D_2$, is a surface with only meridional boundary components (if any), while the other spans $K$.  Since Euler characteristic is additive under Haken sum, we must have two cases: 1) $D_1$ a disc and $D_2$ is a torus, Klein bottle, annulus, or M\"obius band, or 2) $D_1$ is a punctured torus or Klein bottle and $D_2$ is a sphere.

In the first case, $D_1$ is a normal disc spanning $K$ of smaller weight than $D$.  In the second case, Schubert's lemma also says we can suppose that no curve of intersection is separating on both $D_i$'s.  Thus since every curve separates on a sphere, no curve of intersection separates $D_1$.  Pick such a innermost curve on $D_2$ that bounds a disc $E$ of least weight over all innermost curves.  This curve must be 2-sided on $D_1$.  By compressing along $E$, we obtain a disc $D'$ with at most weight of $D$.  If the disc is not normal, then normalization will reduce the weight, contradicting the definition of $D$.

Suppose $D'$ is normal.  Since $D$ is of minimal weight, $D'$ must have the same weight.  This implies that that the weight of $D_2 - E$ equals the weight of $E$.  If there is no other curve of intersection, compressing along the disc $D_2 - \operatorname{Int}(E)$ will result in a non-normal disc and we obtain a contradiction as before.  If there is another curve of intersection, examination shows that there is another innermost curve of intersection on $D_2$ which has less weight than $E$, a contradiction.  
\end{proof}

\section{A bound on the number of elementary moves to unknot}\label{sec:disc_bound}
J. Hass and J. Lagarias  obtained an upper bound of $2^{10^7t}$ on the minimum number of elementary moves to take $K$ to a triangle in one tetrahedron \cite{hass-lagarias2001}. Their key insight was to use normal surface theory to use a normal spanning disc $D$ for the unknot $K$, which was of exponential size (in $t$).  But this disc $D$ was normal with respect to a doubly barycentric subdivision of $M$ with a simplicial neighborhood of $K$ removed.  Before we can move $K$ across the disc $D$, $K$ must first be moved to $\partial D$.  Recall that the bulk of their efforts was on working out how to do this while bounding the number of elementary moves. 

The Hass--Lagarias bound is obtained by isotoping $K$ by elementary moves across an annulus connecting $K$ to a curve on the boundary of the removed neighborhood of $K$, isotoping across the torus boundary of the regular neighborhood to the boundary of a normal disc, and finally isotoping across the normal disc to a single triangle.  Since the number of triangles in the normal disc is at most $2^{8t+6}$,  the chief culprit for their large final bound of $2^{10^7 t}$ is because of their large bound for the first two isotopies.

We will now show how to how to improve the upper bound on the number of moves to $2^{120t+14}$, and subsequently improve the upper bound on the number of Reidemeister moves to unknot an unknot diagram with crossing number $n$ from $2^{10^{11}n}$ to $2^{10^5n}$.  The idea is to use boundary-twisted normal form to implement normal surface theory in the most direct way possible: Theorem~\ref{thm:fund} implies there is a boundary-twisted normal disc $D$ of bounded size.  We straighten out the discs to be piecewise-linear, and move the unknot along $D$ using elementary moves until it becomes a triangle in a tetrahedron.

First, we need to understand how to bound the size of the fundamental normal disc given by Theorem~\ref{thm:fund}.  In \cite{hass-lagarias-pippenger1999} an upper bound was given for the maximal coordinate of any fundamental solution.  This bound depends only on two particular features of the system of integer linear equations: the number of variables, $n$, and the maximum, $m$, over the sum of the absolute values of the coefficients of an equation.  The upper bound is $n \cdot m^\frac{n - 1}{2}$.

In the last column of Table~\ref{big_table}, for each type of truncated tetrahedron, we give the maximum number of normal disc types that share a particular normal arc type on a face.  Since 6 is the max over all tetrahedra types, we see that any matching equation will have at most 12 as the sum of the absolute values of its coefficients.  

Thus plugging these numbers into the bound from \cite{hass-lagarias-pippenger1999}, we obtain $59t\cdot 12^\frac{59t - 1}{2}$.  We get a slightly nicer form by relaxing the bound to $59t \cdot2^{118t-2}$.  

So the total number of discs of a fundamental surface is at most $59t \cdot 59t \cdot2^{118t-2} \leq 2^{120t + 10}$.

This bound with the next basic result will give a bound on the number of elementary moves to deform $K$ to a triangle.  Recall that an \emph{elementary move} of a polygonal link in a piecewise-linear 3-manifold with specified triangulation consists of two kinds of moves (and their inverses) in a tetrahedron (Figure~\ref{elem_moves}): 1) a segment of the link is divided into two by inserting a vertex 2) two segments which are edges of a triangle otherwise disjoint from the link are moved to the third edge of the triangle.

\begin{figure}
\begin{center}
\subfigure{\includegraphics[width=1.8in]{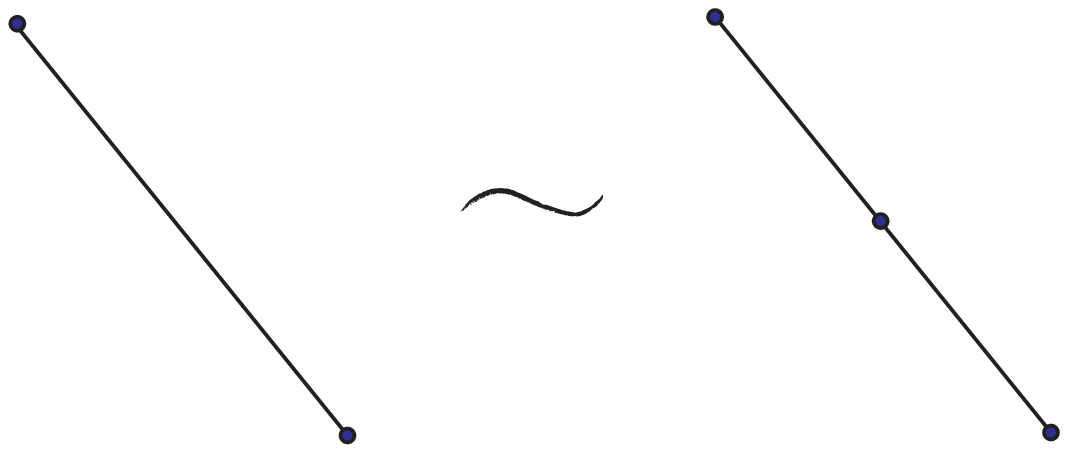}}
\hspace{1.3cm}
\subfigure{\includegraphics[width=2in]{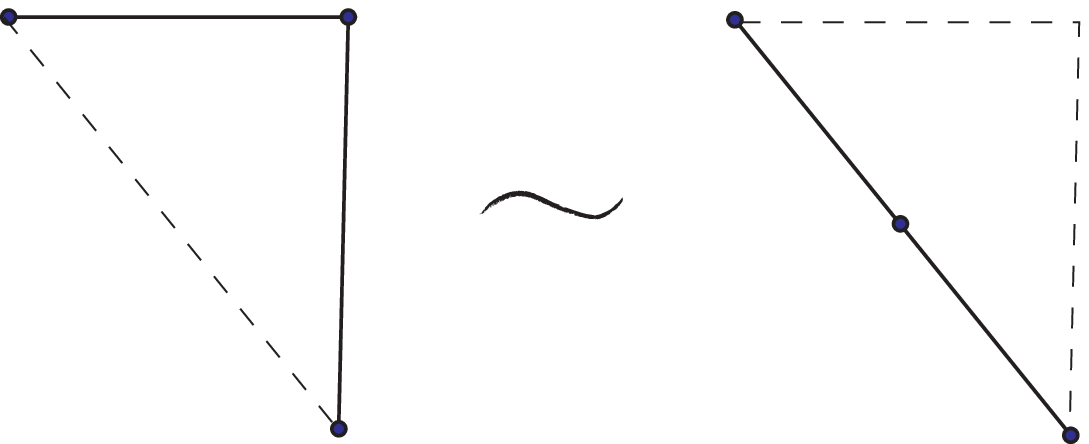}}
\caption{Elementary moves taking place in a tetrahedron}
\label{elem_moves}
\end{center}
\end{figure}

\begin{lem}\cite{hass-lagarias2001}\label{lem:elem_move}
Let $M$ be a triangulated $3$-manifold with $S$ a normal disc in $M$ with $w$ triangles.  Then $\partial S$ can be isotoped to a triangle by a series of at most $2w$ elementary moves in $M$, each of which takes place in a triangle or edge in $S$.
\end{lem}

\begin{thm}[Bounding elementary moves to unknot]\label{thm:main}
Let $M$ be a triangulated $3$--manifold with $t$ tetrahedra and $K$ is a knot in the 1-skeleton with at most one edge in each face.  Suppose $K$ is an unknot, i.e. bounds a disc in $M$.  Then there is a series of at most $2^{120t + 14}$ elementary moves taking $K$ to a triangle lying in a tetrahedron.
\end{thm}

%\begin{rmk}
%The Hass--Lagarias bound is $2^{10^7 t}$.  The larger number results from drilling out a regular neighborhood of the knot.  The knot must be connected by a possibly very long annulus to a curve on the boundary of the neighborhood.  The isotopy must move the knot across the annulus before it can be isotoped across the normal disc in the knot complement.
%\end{rmk}   

\begin{proof}
By Theorem~\ref{thm:btnf}, $K$ bounds a boundary-twisted normal disc $D$, and thus there is a boundary-restricted normal disc $D'$ in the truncated triangulation.  We will suppose $D'$ is of least weight so that Theorem~\ref{thm:fund} applies to $D'$ and obtain a bound of $2^{120t + 10}$ on the number of normal discs in $D'$.  After extending each normal disc to a twisted normal disc, the number of twisted normal discs of $D$ is also bounded by the same number.  Before we can isotope $K$ across $D$ by elementary moves, we need to straighten each twisted normal disc to a piecewise-linear disc.

Almost every twisted normal disc can be straightened to a piecewise-linear disc by straightening any vertex-vertex arcs to become an edge of the tetrahedron.  
The only exception is a ``bigon'', which has boundary composed of exactly two vertex-vertex normal arcs.  

After straightening twisted normal discs and collapsing bigons to edges, a priori, the result may not be embedded.  This can only happen when two of the tetrahedra around an unmarked edge contain twisted normal discs which have vertex-vertex arcs with endpoints on that edge.  Since the disc is not embedded, we can pick two such arcs so that they do not bound a chain of bigons.  The arcs bound a compression disc, and after the compression we have a disc with fewer points of intersection with the 1-skeleton.  Normalizing the disc to boundary-twisted normal form and then truncation will give a boundary-restricted normal disc of lesser weight than $D'$, which is a contradiction. 

By examining twisted normal discs, we see that there are at most 6 sides.  So after straightening, every disc can be divided up into at most 6 piecewise-linear triangles.  Using Lemma~\ref{lem:elem_move} with our bound, we obtain $2(6) \cdot 2^{120t + 10} = 2^{120t + 14}$ as an upper bound on elementary moves.  
\end{proof}

\subsection{Bounding the number of Reidemeister moves}

We recall some results from \cite{hass-lagarias2001}:

\begin{lem}[Triangulating a knot diagram]\label{thm:triangulate_diagram}
Given a knot diagram $D$ of crossing number $n$, there is a triangulated convex polyhedron $P$ in $\mathbb R^3$ with at most $140(n+1)$ tetrahedra so that it contains a knot in the 1-skeleton which orthogonally projects to $D$ on a plane.  Furthermore, each face of a tetrahedron contains at most one edge of the knot.    
\end{lem}

\begin{rmk}
Hass and Lagarias actually get a larger number because they need to assume the knot is in the interior of $P$.
\end{rmk}

\begin{lem}[Reidemeister bound for a projected elementary move]\label{thm:bound_projection}
Let $L$ and $L'$ be polygonal links in $\mathbb R^3$.  Suppose that $L$ (resp. $L'$) has at most $n$ edges and has a link diagram $D$ ( resp. $D'$) under orthogonal projection to the plane $z=0$.  

If $k$ elementary moves take $L$ to $L'$, then at most $2k(n + \frac{1}{2}k + 1)^2$ Reidemeister moves take $D$ to $D'$.  
\end{lem}

Now we can prove

\begin{thm}
Let $D$ be an unknot diagram with $n$ crossings.  Then there is a sequence of at most $2^{10^5 n}$ Reidemeister moves taking $D$ to the standard unknot.
\end{thm}

\begin{proof}
We use Lemma~\ref{thm:triangulate_diagram} to obtain a triangulated polyhedron $P$ of at most $140(n+1)$ tetrahedra such that a knot in its 1-skeleton orthogonally projects to $D$.  Let $t$ be the number of tetrahedra.

By Theorem~\ref{thm:main} there is a sequence of at most $2^{120t + 14}$ elementary moves taking $K$ to a triangle in a tetrahedra.  Using the bound on the projection of an elementary move (Lemma~\ref{thm:bound_projection}) and noting $K$ contains at most $2t$ edges, we obtain the following bound on Reidemeister moves:
\[2^{120t + 14}(2\cdot t + \frac{2^{120t+14}}{2} + 1)^2 \]

This quantity is less than $2^{360t +43} \leq 2^{360\cdot 140(n+1) +43}$.  Except for $n=1$, for which the Reidemeister bound obviously holds, the last is less than $2^{10^5 n}$, which was the desired bound. \end{proof}

\end{document}